\documentclass{amsart}
\usepackage{amssymb}
\usepackage{amsmath}
\usepackage{amsfonts}
\usepackage{pstricks}
\usepackage{pst-node}
\usepackage{graphicx}
\theoremstyle{plain}
\newtheorem{thm}{Theorem}[section]

\newtheorem{lmm}[thm]{Lemma}
\newtheorem{cor}[thm]{Corollary}
\newtheorem{dfn}[thm]{Definition}

\newtheorem{qws}[thm]{Question}

\newtheorem{prb}[thm]{Problem}
\theoremstyle{remark}

\def\pmc#1{\setbox0=\hbox{#1}
    \kern-.1em\copy0\kern-\wd0
    \kern.1em\copy0\kern-\wd0}

\begin{document}

\bigskip

\title[On nerves of fine coverings of acyclic spaces]{On nerves of fine coverings of acyclic spaces} 

\bigskip

\bigskip

\author[U.~H.~Karimov]{Umed H. Karimov}
\address{Institute of Mathematics, 
Academy of Sciences of Tajikistan, 
Ul. Ainy $299^A$, 
734063 Dushanbe, 
Tajikistan}
\email{umedkarimov@gmail.com}

\author[D.~D.~Repov\v{s}]{Du\v san D. Repov\v s } \address{Faculty of Education 
and 
Faculty of Mathematics and Physics,
University of Ljubljana, 
Kardeljeva pl. 16, 
1000 Ljubljana,
Slovenia}
\email{dusan.repovs@guest.arnes.si}



\begin{abstract}
The main results of this paper are:
(1) If a space $X$
can be embedded as a
cellular subspace
of $\mathbb{R}^n$ then $X$
admits arbitrary fine open
coverings whose nerves 
are homeomorphic to the $n-$dimensional
cube $\mathbb{D}^n;$
(2) Every $n-$dimensional cell-like compactum can be
embedded  into $(2n+1)-$dimensional Euclidean space
as a cellular subset; and
(3)  There exists a locally compact planar 
set which is acyclic with respect to {\v C}ech
homology and whose fine coverings are all  nonacyclic.
\end{abstract}

\maketitle

\section{Introduction}

In 1954 Borsuk \cite{Bor}
asked  whether every
compact $ANR$ is homotopy
equivalent to some
compact polyhedron?
His question was
answered in the affirmative in 1977 by West \cite{W}.
Much earlier, in 1928 Aleksandrov \cite{A1928} proved that  
every compact $n-$dimensional
space $X$ admits for any $\varepsilon >0$,
an $\varepsilon-$map onto some
$n-$dimensional finite polyhedron $P$
which is the nerve of some fine
covering of $X$, whereas $X$
does not admit any $\mu-$map
for
some
$\mu >0$ 
onto a polyhedron of dimension less than $n$.
These results
motivated
the  
classical Aleksandrov-Borsuk problem which remains open:

\begin{prb}
Given an  $n-$dimensional compact
absolute neighborhood retract
$X$
and 
$\varepsilon >0$, does there exist
an $\varepsilon-$covering $\mathcal{U}$ of order
$n+1$ such that the natural map of  $X$ onto the nerve
$\mathcal{N}(\mathcal{U})$ of the covering $\mathcal{U}$
induces a homotopy equivalence?
\end{prb}

A special case  is the following, also open problem:

\begin{prb}\label{AR}
Does every
$n-$dimensional
compact absolute retract 
admit a
fine covering
$\mathcal{U}$ of order $n+1$ such that its nerve
$\mathcal{N}(\mathcal{U})$
is contractible?
\end{prb}

Note that for the class
of cell-like cohomology locally conected compacta
the answer 
to
analogous
question is negative, since there exists a
$2-$dimensional cell-like cohomology locally connected compactum
whose fine coverings of order 3 are all nonacyclic \cite{KaUSSR,
KaTajik}.

In the present paper we shall investigate fine coverings of acyclic,
cellular and cell-like spaces. A topological space $X$ is 
called {\it acyclic with respect to {\v
C}ech homology} or simply {\it acyclic} if {\v C}ech homology with
integer coefficients of $X$ is the same as of the point. A {\it
cellular} subspace $X$ of the $n-$dimensional Euclidean space
$\mathbb{R}^n$ is a subspace of $\mathbb{R}^n$ which is the
intersections of a nested system of  $n-$dimensional topological
cubes $D^n$ :
$$X = \bigcap_{i=1}^{\infty}D^n_i, \ \mbox{where} \ D^n_{i+1} \subset {\rm
int} D^n_i.$$ Recall that it follows
by the continuity property of
{\v C}ech homology  that
every cellular space is acyclic.

Our first main result is the following:
\begin{thm}\label{Thm:Main}
If a space $X$ can be embedded as cellular subspace into the
$n-$dimensional Euclidean space $\mathbb{R}^n$ then $X$
admits
arbitrarily fine open coverings whose nerves are all
homeomorphic to the $n-$dimensional cube
$D^n.$
\end{thm}

It is well known that the class of all cellular spaces is quite
large, for example all $n-$dimensional compact cell-like spaces
$X$ can be embedded as cellular subsets of $\mathbb{R}^m,$
provided that $m \geq 2n + 2$ (see e.g. \cite{GS}).

We shall strengthen this fact as follows:
\begin{thm}\label{Thm:Main2}
Every $n-$dimensional cell-like compact space $X$ can be embedded
into
the $(2n+1)-$dimensional Euclidean space
$\mathbb{R}^{2n+1}$
as a cellular subset.
\end{thm}

\begin{cor}\label{contractible}
Every $n-$dimensional cell-like (in particular, contractible)
compact space $X$ admits arbitrarily fine open coverings whose
nerves are all homeomorphic to
the $(2n+1)-$dimensional cube
$D^{2n+1}.$
\end{cor}

Note that there exist $n-$dimensional contractible compacta which
are nonembeddable into $\mathbb{R}^{2n}$ (see e.g.
\cite{KaRe2001}).

Finaly, we shall show that there exist acyclic planar spaces whose
fine covering are all nonacyclic. These spaces are of course not
compact because there is a classical result that any planar
acyclic compactum is cellular (see  e.g. \cite{B,D,M}) and by our Theorem \ref{Thm:Main}, they have fine
coverings whose nerves are all homeomorphic to the $2$-cell $D^2$.

\section{Preliminaries}

We shall begin by fixing some terminology and notations and we
shall give some definitions which will be used in the sequel. All
undefined terms can be found in \cite{B,D,E,H,HW,
SeTh,SE}.

By a {\it covering} $\mathcal{U}$ of $X$ we mean a system of open
subsets of a metric space $X$ whose union  is $X$. If the
space $X$ is compact then by a {\it covering} we mean a finite covering.
We consider the standard metric $\rho$ on the Euclidean space
$\mathbb{R}^n$ and its subspaces. For a subspace $A$ of the space
$X$ and for a positive number $d$  we denote the $d-$neighborhood
of the set $A$ in $X$ by $N(A, d)$, i.e. $$N(A, d) = \{x: x\in X \
\mbox{and}\  \rho(x, A) < d\}.$$ In particular, the open ball
$B(a, d)$ in a metric space with center at the point $a$ and
radius $d$ is the set $N(\{a\}, d)$. By the {\it mesh} of the
covering $\mathcal{U},$\ \ ${\rm mesh}(\mathcal{U}),$ we mean the
supremum of the diameters of all elements of the covering
$\mathcal{U}$. We say that the space admits {\it  fine acyclic}
coverings if for every open covering $\mathcal{U}$ there exists a
refinement $\mathcal{V}$ of $\mathcal{U}$ such that homology of
the nerve $\mathcal{N}(\mathcal{V})$ is trivial, i.e. homology of
$\mathcal{N}(\mathcal{V})$ is the same as homology of the point.
For compact spaces this is equivalent to existence of acyclic
coverings $\mathcal{V}$ with mesh$(\mathcal{V}) < \varepsilon,$
for every positive number $\varepsilon$.

We consider only {\v C}ech homology with integer coefficients.

\begin{dfn}
A {\rm kernel} $U_i^0$ of the element $U_i$ of the covering
$\mathcal{U} = \{U_i\}_{i=\overline{1,n}}$ of the space $X$ is an
open non-empty subset of $U_i$ such that it does not intersect
with other elements $U_j, j \neq i,$ of the covering
$\mathcal{U}.$
\end{dfn}

\begin{dfn}
A covering $\mathcal{U} = \{U_i\}_{i=\overline{1,n}}$ is called
{\rm canonical} on the subspace $A\subset X$ if for every $i$ such
that $U_i\cap A \neq \emptyset$ it follows that $U_i^0\cap A \neq
\emptyset.$
\end{dfn}

\begin{dfn}
A canonical covering $\mathcal{U} = \{U_i\}_{i=\overline{1,n}}$
on the subspace $A$ is called a {\rm canonization} of the covering
$\mathcal{V} = \{V_i\}_{i=\overline{1,n}}$ if $U_i \subset V_i$
for every $i,$ and this refinement induces a simplicial
isomorphism between the nerves $\mathcal{N}(\mathcal{U})$ and
$\mathcal{N}(\mathcal{V}).$
\end{dfn}

\begin{lmm}\label{canonization}
For every finite covering $\mathcal{V} =
\{V_i\}_{i=\overline{1,n}}$ of a metric space $X$ and its subset
$A$ without isolated points there exists a canonization
$\mathcal{U}$ of the covering $\mathcal{V}$ on the subspace $A.$
\end{lmm}

\begin{proof}
In every nonempty intersection $V_{i_0}\cap V_{i_0}\cap\cdots
V_{i_k}\cap A$ let us choose a point $a_{i_0i_1\dots i_k}$ such
that to different systems of open sets there correspond different
points (this is possible because the set $A$ does not contain
isolated points). We get a finite set of points. 

Let $d$ be any
positive number less than the minimum of the distances between the
chosen points and such that if the intersection $V_i\cap A \neq
\emptyset$ then $B(a_i, d)\subset V_i.$ Let $U_i = V_i \setminus
\bigcup \overline{B(a_j,\frac{d}{2}})$ (the union is over all $j,
\ j \neq i$ for which the point $a_j$ is defined, i.e. $V_j\cap A
\neq \emptyset$). The nerves of the coverings $\mathcal{U}$ and
$\mathcal{V}$ are isomorphic since we did not remove the points
$a_{i_0i_1\dots i_k}$ from the space  $X$ and since the balls
$B(a_i, \frac{d}{2})$ lie in the kernel of the $U_i.$ 

Therefore
the covering $\mathcal{U}$ is a canonization of the covering
$\mathcal{V}$ on the subspace $A.$
\end{proof}

\begin{dfn}
By a {\rm chain connecting the element} $U$ of the covering
$\mathcal{U}$ with subset $A$ along the connected subset $M$ of
the space $X$ we mean a system $\{U_{i_1}, U_{i_2},\dots
U_{i_m}\}$ of elements of $\mathcal{U}$ such that $U_{i_1} = U,$
$U_{i_k}\cap A = \emptyset$ for $k < m,$ $U_{i_m}\cap A \neq
\emptyset$  and $U_{i_t}\cap U_{i_{t+1}}\cap M \neq \emptyset$ for
$t = \overline{1, (m-1)}.$
\end{dfn}

Next,
we shall
need the following construction. Consider the covering
$\mathcal{U}$ of a topological space $X$ and the 4-tuple $\{U, x,
\varepsilon, m\}$ in which $U\in \mathcal{U},$ $x\in U^0$ ($U^0$
is kernel of the $U),$ $\varepsilon$ is a positive number such
that $B(x, \varepsilon) \subset U^0,$ and $m$ is any natural
number. Consider the covering $\mathcal{U}'$ which consists of all
elements of the covering $\mathcal{U}$ except the element $U.$
Instead of $U$ we choose the following $m$ elements for $m>1$:
\begin{itemize}
    \item $U(x, \varepsilon, m, 1) = U\setminus \overline{B(x, \frac{\varepsilon}{2})}$,
    \item $U(x, \varepsilon, m, k) = B(x, \frac{\varepsilon}{k-1})\setminus \overline{B(x,
    \frac{\varepsilon}{k+1})},$ for $k > 1, k < m,$
    \item $U(x, \varepsilon, m, m) = B(x,
    \frac{\varepsilon}{m-1}).$
\end{itemize}
If $m = 1,$ then $U(x, \varepsilon, 1, 1) = U.$ The covering
$\mathcal{U}'$ is called the {\it grating} of $\mathcal{U}$ with
respect to the 4-tuple $\{U, x, \varepsilon, m\}.$

If the element $U$ of the covering $\mathcal{U}$ is connected then
$\mathcal{N}(\mathcal{U}') = \mathcal{N}(\mathcal{U})\cup P,$
where $P$ is a segment subdivided into $m-1$ parts if $m>1$ and it
is the empty set if $m=1.$

Let $\mathcal{U}$ be any covering of the space $X$ which refines a
covering $\mathcal{W}$ and let $\varphi : \mathcal{N}(\mathcal{U})
\to \mathcal{N}(\mathcal{W})$ be a simplicial mapping induced by
this refinement. Suppose that $\{U_{i_0}, U_{i_1},\dots,
U_{i_m}\}$ are subsets of the covering $\mathcal{U}$ having empty
intersection.

\begin{dfn}\label{dfn:extension}
The covering $\mathcal{W}$ is called an {\rm extension} of the
covering $\mathcal{U}$ with respect to the set $\{U_{i_0},
U_{i_1},\dots, U_{i_m}\}$ if the mapping $\varphi$ is injective
and the complex $\mathcal{N}(\mathcal{W})$ is the union of the
complex $\mathcal{N}(\mathcal{U})$ with an $m-$dimensional simplex
corresponding to $\{U_{i_0},U_{i_1},\dots, U_{i_m}\}$ and possibly
some of its faces.
\end{dfn}

\begin{lmm}\label{extension}
For every covering $\mathcal{U}$ canonical on the set $A$ and for
every system of its elements $\{U_{i_0}, U_{i_1},\dots U_{i_m}\}$
such that $\bigcap_{t=0}^{m}U_{i_t} = \emptyset$ and $A\cap
U_{i_t}\neq \emptyset$ there exists for every $t,$ a covering
canonical on $A$ which is an extension of the covering
$\mathcal{U}$ with respect to $\{U_{i_0}, U_{i_1},\dots
U_{i_m}\}.$
\end{lmm}

\begin{proof}
In the kernel of one of the sets $U_{i_0}, U_{i_1},\dots U_{i_m}$
choose a ball $B(a, d)$ and replace $U_{i_t}$ by $U_{i_t}\cup B(a,
d)$ for every $t =\overline{0,m}.$ We obviously get the desired
extension.
\end{proof}

Subspace $X$ of $\mathbb{R}^n$ is cellular if and only the
quotient space $\mathbb{R}^n/X$ is homeomorphic to $\mathbb{R}^n$
(see  e.g. \cite{GS}). Therefore we can assume that for a
cellular subset $X,$ $X = \bigcap_{i=1}^{\infty}D^n_i, \
\mbox{where} \ D^n_{i+1} \subset {\rm int} D^n_i,$ there exists a
retraction $r_i: D^n_{i} \to D^n_{i+1}$  such that preimage of
every point $x$ of the boundary $\partial D^n_{i+1}$ is
homeomorphic to the segment $[0, 1].$

\begin{dfn} {\rm (}\cite{GS, H}{\rm )}. A polyhedral neighborhood $N$ of the polyhedron $P \subset
\mathbb{R}^n$  is called {\rm regular} if there exists a piecewise
linear mapping $\varphi: N\times [0, 1] \to N,$ such that
$\varphi(x, 0) = x,$  $\varphi(x, 1)\in P $ for all $x\in N$ and
$\varphi(x, t) = x$ for $x\in P$ and $t\in [0,1]$ or in other
words, $P$ is a strong deformation retract of $N$ under a
piecewise linear homotopy $\varphi.$
\end{dfn}

\begin{dfn} {\rm (}\cite{GS, H}{\rm )}. An {\rm $\varepsilon$ push of the pair $(\mathbb{R}^n, X)$}
is a homeomorphism $h$ of $\mathbb{R}^n$ to itself for which there
exists a homotopy $\varphi: \mathbb{R}^n\times [0,1]\to
\mathbb{R}^n$ such that
\begin{enumerate}
    \item $\varphi(x,0) = x,\ \varphi(x,1) = h(x)$;
    \item $\varphi_t:\mathbb{R}^n \to \mathbb{R}^n$ is a homeomorphism for every $t\in
    [0,1]$ and $\rho(x, \varphi_t(x) )< \varepsilon$ for all $x\in
    \mathbb{R}^n;$
    \item $\varphi(x,t) = x$ for every $t\in [0, 1]$ and all $x$
    such that $\rho(x, X) \geq \varepsilon.$
\end{enumerate}
\end{dfn}

\begin{dfn} {\rm (}\cite{GS, H}{\rm )}.
Let $P$ be a compact subpolyhedron of $\mathbb{R}^n$ and let
$\varepsilon$ be a positive real number. {\rm An
$\varepsilon-$regular neighborhood} of $P$ in $\mathbb{R}^n$ is a
regular neighborhood $N$ of $P$ such that for any compact subset Y
of $\mathbb{R}^n\setminus P$, there is an $\varepsilon-$push $h$
of $(\mathbb{R}^n, P)$ such that $h(Y) \cap N = \emptyset.$
\end{dfn}

We note that it follows by definition that every
$\varepsilon-$regular neighborhood $N$ of subpolyhedron $P$ is a
proper subset of $N(P, \varepsilon).$
\medskip

The following lemma follows by the regular neighborhood theory
(see  e.g. \cite{GS, H}).

\begin{lmm}\label{RegularNeighborhood}
For any finite subpolyhedron $P$ of $\mathbb{R}^n$ and any
$\varepsilon > 0$ there exists an $\varepsilon-$regular
neighborhood $N$ of $P.$
\end{lmm}

\section{Proof of Theorem~\ref{Thm:Main} }

Since $X$ is a cellular subset of $\mathbb{R}^n$ we have $$X =
\bigcap_{i=1}^{\infty}D^n_i, \ \mbox{where} \ D^n_{i+1} \subset
{\rm int} D^n$$ and there are natural retractions $r_i: D^n_i \to
D^n_{i+1}.$

Fix a positive number $\varepsilon$ and some natural number $K$
which will be specified later. Since $X \subset {\rm int}\ D_1$
and $X$ is a compact space there exists a finite system of open
balls of radius $\varepsilon' < \frac{\varepsilon}{K}$ in
$\mathbb{R}^n$ which cover $X,$ i.e. $X \subset \cup_{x \in F}
B(x, \varepsilon'),$ $F$ is a finite subset of $X$ for which
$\cup_{x \in F} B(x, \varepsilon')\subset D^n_1.$ There exists an
index $i_0$ such that $D_{i_0}^n \subset \cup_{x\in F} B(x,
\varepsilon').$ Let $r$ be a natural retraction of $D_1^n$ on
$D_{i_0}^n.$  Since the mapping $r$ is uniformly continuous there
exists a positive number $\delta < \varepsilon',$ such that
$\varrho (r(x), r(y)) < \varepsilon'$ whenever $\varrho (x, y) <
\delta.$ 

Consider a triangulation of $D_1^n$ with the diameters of
simplices less than $\frac{\delta}{2}.$ Consider a covering of
$D_1^n$ by open stars of the vertices of this triangulation.
According to Lemma \ref{canonization} there exists a refinement
$\mathcal{U} = \{U_i\}_{i=\overline{1,m}}$ of this covering which
is canonical on $X$ . Note that the nerve
$\mathcal{N}(\mathcal{U})$ is homeomorphic to $D^n.$

We wish to associate to every open set $U_i$ of the $\mathcal{U}$
some open subset of the space $X.$ If the intersection $U_i\cap X$
is nonempty then we associate to $U_i$ the open set $U_i\cap X$ in
$X.$ If $U_i\cap X = \emptyset$ then we choose the point $y_i \in
U_i.$ The point $r(y_i)$ belongs to some ball $B(x_{i},
\varepsilon'),$ $x_i\in F,$ and the subspace $r^{-1}(r(y_i))$ is
homeomorphic to a segment if $y_i \notin D_{i_0}$ or is a point if
$y_i \in D_{i_0}.$ So the union $M_i = r^{-1}(r(y_i)) \cup B(x_i,
\varepsilon')$ is a connected set. 

Since $M_i$ is connected there obviously
exists  a chain $\{U_{i_1}, U_{i_2},\dots U_{i_{m(i)}}\}$
connecting $U_i$  with $X$ along $M_i.$ Since the covering
$\mathcal{U}$ is canonical on $X,$ the intersection of the kernel
of $U_{i_{m(i)}}$ and $X$ is nonempty, and we can find a point
$z_i \in U_{i_{m(i)}}^0\cap X.$ Let $\varepsilon_i$ be a positive
number such that $B(z_i, \varepsilon_i)\cap X\subset
U_{i_{m(i)}}^0\cap X.$ So we have a 4-tuple $\{U_{i_{m(i)}}\cap X,
z_i, \varepsilon_i, m(i)\}$ and we can take a grating of
$\mathcal{U}|_X = \{U_i\cap X\}_{i=\overline{1,m}}$ with respect
to this 4-tuple.

Repeat this procedure for all $i.$ We get some canonical covering
$\mathcal{U}'$ of $X.$ There is a simplicial mapping $\mathcal{J}:
\mathcal{N}(\mathcal{U}')\to \mathcal{N}(\mathcal{U})$ which maps
the vertices $U(z_i, \varepsilon_i, m(i), k)$ to the vertices
$U_{i_k}.$ This mapping is in general not injective because some
element $U_k\in \mathcal{U}$ can be the element of several chains
of the type $\{U_{i_1}, U_{i_2},\dots U_{i_{m(i)}}\}.$ 

Consider a
new covering $\mathcal{W}$ whose 
elements are unions of all
elements of $\mathcal{U}'$ which correspond to the elements $U_k$
under a mapping $\mathcal{T}.$

Let us estimate the diameters of the elements of $\mathcal{W}$.
Take two points $a_1$ and $a_2$ from any $W \in \mathcal{W}$. By
construction, there must exist two sets $M_{i_1}$ and $M_{i_2}$
which intersect with $U_k.$ The distance between the points
$r(y_{i_1})$ and $r(y_{i_2})$ is less than $\varepsilon'$ since
the diameter of $U_k$ is less than $\delta.$ Diameters of balls
$B(x_{i_1}, \varepsilon')$ and $B(x_{i_2}, \varepsilon')$ are less
than or equal to $2\varepsilon'.$  The diameters of the elements
of the covering $\mathcal{U}$ are also less than $\varepsilon'$
since $\delta < \varepsilon'.$ 

By the Triangle Inequality it
follows that $\rho (a_1, a_2) < 7\varepsilon'$ therefore
diam$(\mathcal{W})\leq 7\varepsilon'$. We now have the injective
mapping $\mathcal{J}': \mathcal{N}(\mathcal{W}) \to
\mathcal{N}(\mathcal{U})$ which maps vertices of
$\mathcal{N}(\mathcal{W})$ bijectively onto the vertices of
$\mathcal{N}(\mathcal{U}).$ Suppose that $\mathcal{J}'$ is not
surjective. Then there exists a system of elements $\{W_{i_1},
W_{i_2},\dots W_{i_{m(i)}}\}$ of the covering $\mathcal{W}$ with
empty intersection. 

Let us apply the operation of the extension of
the covering $\mathcal{W}$ (see Definition \ref{dfn:extension}).
We get a new covering. Since all coverings are finite, after few
applications of this operation we finally get the covering
$\mathcal{W'}$ of the space $X$ and a bijective mapping
$\mathcal{J}'': \mathcal{N}(\mathcal{W'}) \to
\mathcal{N}(\mathcal{U}).$

Let us estimate the distance between the points of the sets
$W_{i_k}$ and $W_{i_l}$ which are the vertices of same simplex of
the polyhedron $\mathcal{N}(\mathcal{W'}).$ For the mapping
$\mathcal{T}'$ to $W_{i_k}$ and $W_{i_l}$ there correspond two
elements $U_{i_k}$ and $U_{i_l}$ which intersect. We have the
points $y_{i_k}\in U_{i_k}$ and $y_{i_l}\in U_{i_l}.$ Since
$U_{i_k}\cap U_{i_l} \neq \emptyset,$ we have $\rho(y_{i_k},
y_{i_l})< 2\delta$ and $\rho(r(y_{i_k}), r(y_{i_l}))<
2\varepsilon'.$ By construction $\rho(r(y_{i_k}), W_{i_l}) <
7\varepsilon'$ for every $k.$ 

It follows that the distance between
any points of $W_{i_k}$ and $W_{i_l}$ is less than
$16\varepsilon'.$ So the diameters of the elements of the covering
$\mathcal{W'}$ are no more than $16\varepsilon'.$ Since the number
$K$ was arbitrary we can put $K > 16$ and
get that 
${\rm
diam}(\mathcal{W'})< \varepsilon.$ So we have a fine covering
whose nerve is homeomorphic to $D^n.$

\section{Proof of Theorem~\ref{Thm:Main2} }

We shall need the following lemmas.
\begin{lmm}\label{Freudenthal}{\rm (Freudenthal \cite{E, F, S}).}
Every compact metrizable space $X$ is homeomorphic to the inverse
limit of the inverse sequence $\{P_i \stackrel{\
f_i}\longleftarrow P_{i+1}\}_{i \in \mathbb{N}}$ of finite
polyhedra $P_i$ with piecewise linear {\rm (}i.e. quasi-simplicial
\cite[pp.148, 153]{E}{\rm )} surjective projections $f_i.$ If
${\rm dim}\ X \leq n$ then ${\rm dim}\ P_i \leq n.$
\end{lmm}

\begin{lmm}\label{Pseudo-polyhedral}{\rm (see \cite[Theorem 3.3]{GS}).}
Let ${\rm dim} X \leq n$ and suppose that $X$ is homeomorphic to
the inverse limit of the inverse sequence $\{P_i \stackrel{\
f_i}\longleftarrow P_{i+1}\}_{i \in \mathbb{N}}$ of finite
polyhedra $P_i$ with piecewise linear surjective projections $f_i$
and ${\rm dim}\ P_i \leq n.$ Then for every $i,\ P_i$ can be
embedded as subpolyhedron $R_i$ in $\mathbb{R}^{2n+1}$ so that:
\begin{itemize}
    \item For every $i$ there exist a $q_i-$regular neighborhood $N_i$ of $R_i$ in
$\mathbb{R}^{2n+1}, q_i < \frac{1}{i}$ and $\overline{N}_{i+1}
\subset N_{i};$
    \item $X$ is homeomorphic to $\cap_{i=1}^{\infty}N_{i}.$
\end{itemize}
\end{lmm}

Let us give a brief sketch of the proof of this lemma, see also
\cite[Exercise 3.4.5]{DV} .
\begin{proof}
Let $K_1$ be the number of vertices of the polyhedron $P_1$ for
some fixed triangulation. Choose the points $\mathcal{P}_1 =
\{p_{1,1},p_{1,2},\dots p_{1,K_1}\}$ in general position in the
space $\mathbb{R}^{2n+1}$ and embed simplicially the polyhedron
$P_1$ in $\mathbb{R}^{2n+1}$ in such a way that to the vertices of
$P_1$ there correspond the points $\mathcal{P}_1$ (see  e.g.
\cite{HW, SeTh}). Denote by $Q_1$ the image of polyhedron $P_1$ in
$\mathbb{R}^{2n+1}$ with a given triangulation.

Since $f_1$ is a quasi-simplicial mapping there exists barycentric
triangulations of $P_1$ and $P_2$ such that $f_1$ becomes
simplicial mapping. Let $L_1$ and $K_2$ be the number of vertices
of the polyhedra $P_1$ and $P_2$ after these triangulations,
respectively. We have points $\{q_{1,1}, q_{1,2},\dots
q_{1,L_1}\}$ of the polyhedron $Q_1$ which correspond to the
vertices of the polyhedra $P_1$ for this triangulation. 

These
points are not in general position but we can move them in such a
way that we get the points $\mathcal{R}_1 = \{r_{1,1},
r_{1,2},\dots r_{1,L_1}\}$ which are in general position and we
get a new subpolyhedron $R_1$ of $\mathbb{R}^n$ piecewise
homeomorphic to $Q_1$ generated by these points with the
simplicial mapping $f_1:P_2 \to R_1$ (here and in the sequel we
shall use the same symbol for mappings if the domain/range are the
same and if the corresponding diagram is commutative). 

Let $N_1$
be $1-$regular neighborhood of the $R_1,$ see Lemma
\ref{RegularNeighborhood}. Let $r_1$ be the distance between $R_1$
and $\mathbb{R}^{2n+1}\setminus N_1.$

Let $d_1$ be any positive number less than the number 1 and the
maximum of the diameters of the simplices of the polyhedron $R_1.$
Choose the points $\mathcal{P}_2 = \{p_{2,1},p_{2,2},\dots
p_{2,K_2}\}$ in $\mathbb{R}^{2n+1}$ which satisfy the following
conditions:
\begin{enumerate}
    \item All points $\mathcal{R}_1\cup \mathcal{P}_2$ are in
    general position, see. e.g. \cite[p. 102, Theorem 1.10.2]{E};
    \item If the vertex corresponding to the point $p_{2,i}$ is mapped by $f_1$ to the point
    $r_{1,j}$ then $p_{2,i}\in B(r_{1,j},\ \min\{\frac{r_1}{3},
    \frac{d_1}{3}\})$.
\end{enumerate}

Since the points $\mathcal{P}_2$ are in general position we can
simplicially embed  the polyhedron $P_2$ with respect to these
vertices. Call by $Q_2$ the image of $P_2$ in $\mathbb{R}^{2n+1}.$
Let us estimate the distance between the points $x\in Q_2$ and
$f_1(x)\in R_1.$ Take any point $x\in P_2.$ Then we have for some
$\lambda_i,\ \sum \lambda_i = 1,\ \lambda_i \geq 0$ and for some
$p_{2,i}$ that $x = \sum \lambda_ip_{2,i}$ where $p_{2,i}$ are
vertices of some simplex of the polyhedron $P_2$ which contains
$x.$ Then $f(x) = \sum \lambda_if(p_{2,i}).$ 

Further, $$\rho(x,
f(x)) = ||x-f(x)|| = ||\sum \lambda_i(p_{2,i}-f(p_{2,i}))|| <$$  
$$ < \sum
\lambda_i\cdot \min\{\frac{r_1}{3}, \frac{d_1}{3}\} =
\min\{\frac{r_1}{3}, \frac{d_1}{3}\} = \delta_1.$$ 
It follows that
$N(Q_2, \delta_1) \subset N_1.$

So we have a triad $\{R_1, N_1, Q_2\}$ of subpolyhedra of
$\mathbb{R}^{2n+1}$ such that $R_1,$ is piecewise homeomorphic to
$P_1,$ polyhedron $N_1$ is $1-$regular neighborhood of $R_1,$
$Q_2$ is homeomorphic to $P_2.$ There is a natural mapping $f_1:
Q_2\to R_1 $ which is associated with $f_1:P_2 \to P_1$ and for
any $x\in Q_{2}$ we have 
$$\rho(x, f_1(x))\leq \delta_1 =
\min\{\frac{r_1}{3}, \frac{d_1}{3}\}$$
where $r_1$ is the distance
between $R_1$ and $\mathbb{R}^{2n+1}\setminus N_1$ and $d_1$ is
maximum of the diameters of the simplices of the polyhedron $R_1.$

Let us suppose that we are given for some index $i$ the triad
$\{R_i, N_i, Q_{i+1}\}$ of subpolyhedra of $\mathbb{R}^{2n+1}$
such that:
\begin{itemize}
    \item $R_i$ is piecewise homeomorphic to $P_i;$
    \item $N_i$ is $q_i-$regular neighborhood of $R_i,$ \ $q_i < {\rm
min}\{\frac{1}{i}, d_i\}$ and $d_i$ is maximum of the diameters of
the simplices of the polyhedron $R_i;$
    \item $Q_{i+1}$ is homeomorphic to $P_{i+1}$ and there is a natural mapping $f_{i}: Q_{i+1}\to R_i$ which is
associated with $f_{i}:P_{i+1} \to P_i;$
    \item for any $x\in Q_{i+1}$ we have $$\rho(x, f_i(x))\leq \delta_i = \min\{q_i,
\frac{r_i}{3}, \frac{d_i}{3}\}$$
where $r_i$ is the distance
between $R_i$ and $\mathbb{R}^{2n+1}\setminus N_i.$
\end{itemize}

We call the triad with this properties a {\it special triad}.

Now we construct the special triad $\{R_{i+1}, N_{i+1}, Q_{i+2}\}$
in the following way. We have a piecewise linear mapping
$f_{i+1}:P_{i+2}\to P_{i+1} = Q_{i+1}$ therefore there exist
barycentric subdivisions of $P_{i+2}$ and $Q_{i+1}$ such that
$f_{i+1}$ becomes a simplicial mapping. Let $L_{i+1}$ and
$K_{i+2}$ be the number of vertices of the polyhedra $P_{i+1}$ and
$P_{i+2}$ after these triangulations respectively. We have points
$\{q_{i+1,1}, q_{i+1,2},\dots q_{i+1,L_{i+1}}\}$ of the polyhedron
$Q_{i+1}$ which correspond to the vertices of the polyhedra
$P_{i+1}$ for this triangulation. 

We move these points in such a
way that we get the points 
$$\mathcal{R}_{i+1} = \{r_{i+1,1},
r_{i+1,2},\dots r_{i+1,L_{i+1}}\}$$ 
which are in general position
and $\rho(q_{i+1,i}, r_{i+1,i}) < \frac{r_i}{3}.$ We get a new
polyhedron $R_{i+1}$ which lies in the neighborhood $N(Q_{i+1},
\frac{r_i}{3})$ and for which we have a simplicial mapping
$f_{i+1}:P_{i+2} \to R_{i+1}.$ Let $d_{i+1}$ be maximum of the
diameters of the simplices of the polyhedron $R_{i+1}.$ Let
$q_{i+1}< {\rm min} \{\frac{1}{i+1}, d_{i+1}\}$ be a such number
that $q_{i+1}-$regular neighborhood $N_{i+1}$ of $R_{i+1}$ be
subset of $N(Q_{i+1}, \frac{r_i}{3}).$ 

Let $r_{i+1}$ be the
distance between $R_{i+1}$ and $\mathbb{R}^{2n+1}\setminus
N_{i+1}$ and let $d_{i+1}$ be maximum of the diameters of the
simplices of the polyhedron $R_{i+1}.$ Choose the points
$\mathcal{P}_{i+2} = \{p_{i+2,1},p_{i+2,2},\dots p_{i+2,K_i+2}\}$
in $\mathbb{R}^{2n+1}$ satisfying the following conditions:
\begin{enumerate}
    \item All points $\cup_{i=1}^{i+1}\mathcal{R}_i\cup \mathcal{P}_{i+2}$ are in
    general position;
    \item If the vertex corresponding to the point $p_{i+2,i}$ is mapped by $f_{i+1}$ to the point
    $r_{i+1,j}$ then $p_{i+2,i}\in B(r_{i+1,j},\ \min\{q_{i+1}, \frac{r_{i+1}}{3},
    \frac{d_{i+1}}{3}\})$.
\end{enumerate}

Since the points $\mathcal{P}_{i+2}$ are in general position we
can simplicially embed the polyhedron $P_{i+2}$ with respect to
these vertices. Call by $Q_{i+2}$ the image of $P_{i+2}$ in
$\mathbb{R}^{2n+1}.$ It is easy to see that $N(Q_{i+2},
\frac{r_{i+1}}{3}) \subset N_{i+1}.$ And we have a special triad
$\{R_{i+1}, N_{i+1}, Q_{i+2}\}.$ By induction we now have the
special triad $\{R_k, N_k, Q_{k+1}\}$ for every $k\in \mathbb{N}.$

If we consider two different sequences of points $x_i \in P_i,
f_i(x_{i+1}) = x_i$ and $x_i' \in P_i, f_i(x_{i+1}') = x_i',$ then
obviously there exists an index $i_0$ such that $x_i$ and $x_i'$
belong
to different simplices
of $R_{i_0}.$ It follows by our
choice
of numbers $d_i$ that the limit points $x$ and $x'$ of
the sequences $\{x_i\}$ and $\{x_i'\}$ are different. 

Therefore
the space $X$ is homeomorphic to the intersection
$\bigcap_1^{\infty}\overline{N_i}$ and is the limit of the
sequence of polyhedra $\{R_i\}_{i \in \mathbb{N}}.$
\end{proof}

Now we can prove 
Theorem \ref{Thm:Main2}. First let us prove Theorem
\ref{Thm:Main2} in the case $n = 1,$ i.e. let us prove that every
$1-$dimensional cell-like compactum $X$ can be embedded as a
cellular subspace in $\mathbb{R}^3.$ 

According to the
Case-Chamberlin theorem, every $1-$dimensional cell-like continuum
is tree-like, i.e. any open covering has a tree-like refinement (a
refinement whose
nerve is
a $1-$dimensional finite
contractible complex)
\cite{CC}. By the proof of the Freudenthal
Theorem \cite{S} it follows that $X =
\underleftarrow{\lim}(P_i\stackrel{f_i}\longleftarrow P_{i+1}),$
where each $P_i$ is a contractible $1-$dimensional polyhedron and
all projections $f_i$ are piecewise linear mappings. 

It follows
by Lemma
\ref{Pseudo-polyhedral}  that $X$ can be embedded in
$\mathbb{R}^3$ so that its image has arbitrary fine neighborhoods
$N_i$ with contractible spines $R_i \approx P_i,$ i.e. $N_i$ is
homeomorphic to $D^3$ and the embedding of $X$ in $\mathbb{R}^3$
is cellular
in this case.

Let now $n \geq 2.$ Then we embed space $X$ in the regular way
in $\mathbb{R}^{2n+1}$ according to  Lemma
\ref{Pseudo-polyhedral}. Let us show that such
an
embedded space $X$
satisfies the cellularity criterion in $\mathbb{R}^{2n+1},$ (see
\cite{DV}). Consider any neighborhood $U$ of 
$X$ in
$\mathbb{R}^{2n+1}.$ 

Since the space $X$ is cell-like there exists
a neighborhood $V$ of $X$ in $U$ such that the embedding $V
\subset U$ is homotopic to the constant mapping. Consider any
mapping $f$ of $\partial D^2$ to $V\setminus X.$ Since the
embedding $V \hookrightarrow U$ is homotopic to the constant
mapping there exists an extension  $\overline{f}: D^2 \to U.$ By
the Simpicial Approximation Theorem we can suppose that $f$ and
$\overline{f}$ are simplicial mappings and the image of
$\overline{f}$ is a $2-$dimensional polyhedron in $U.$ 

Let
$\varepsilon$ be a
positive number such that $N(X,
\varepsilon)\subset V$ and let us choose 
index $i$ so that the
$q_i-$regular neighborhood $N_{i}$ of $R_{i},$  is a subset of $V$
(it suffices to require $\ q_i < \varepsilon.$) We may assume
that
${\rm Im} \overline{f}$ and $R_{i}$ are in general position and
spaces ${\rm Im}\overline{f}$ and $P_i$ do not intersect, ${\rm
Im}\overline{f} \cap P_i = \emptyset$, since $2+n < 2n+1.$ 

By Lemmas \ref{RegularNeighborhood} or
\ref{Pseudo-polyhedral} there exist
for
$N_{i}$, a $q_i-$push $h_{i}$ of the pair
$(\mathbb{R}^{2n+1}, R_{i})$ such that $h_{i}\overline{f}(D^2)\cap
N_{i} = \emptyset.$ It follows that $h_{i}\overline{f}(D^2)\cap X
= \emptyset.$ Therefore $f:\partial D^2\to U\setminus X$ is
inessential and we obtain
an embedding of the
cell-like space $X$ into
$\mathbb{R}^{2n+1}$, for $2n+1\geq 5$, which satisfies the
cellularity criterion of McMillan (cf. \cite{McM} or \cite[Theorem
3.2.3]{DV}). It follows that $X$ is cellular.


\section{Acyclic subspaces of the plane whose
fine coverings are all nonacyclic}

We shall present two examples of locally compact planar acyclic 
with respect to {\v C}ech homology spaces whose fine coverings are all nonacyclic.
\medskip

{\bf Example 5.1.} Consider in the plane $\mathbb{R}^2$ a
countable bouquet of circles $S^1_i$,
with a base point $A$
and with a common tangent line,
whose diameters
 tend to infinity.
From every
circle $S^1_i$ remove a small open arc 
$AA_i$ such
that the diameters of these arcs tend to 0. We get the desired
space $X_1$, see Figure 1.

\begin{center} 
\includegraphics*[trim=0cm 5cm 0cm -1cm, clip=true, width=8cm]{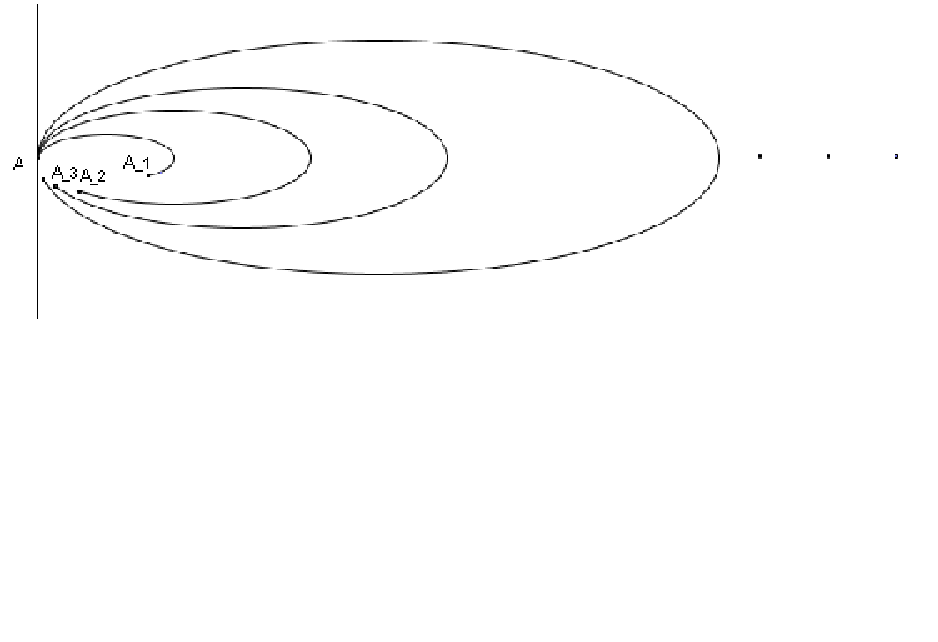}

\begin{center}\vspace{-3mm}\normalsize{Fig. 1. Locally compact acyclic
space whose fine coverings are all nonacyclic.}\end{center}
\end{center}

Obviously $X_1$ is locally compact space. Consider the following
cofinite  system of coverings of $X_1.$ Triangulate the segments
$S_{i}^1 \setminus  AA_{i}$ and take their coverings by
open stars of all vertices of the triangulations except the stars
of the vertex $A.$ For the point $A$ consider the open set $B(A,
\varepsilon)\cap X_1.$ Obviously, the coverings of such type are
cofinal in the set of all coverings of $X_1.$ 

The nerves of this
coverings are homeomorphic to the countable bouquet
$$(\vee_1^nI_i)\bigvee_A (\vee_{n+1}^{\infty}S_i^1)$$ of circles and
a finite number of segments with respect to the point $A.$
Therefore their $1-$dimensional homology groups are isomorphic to
the direct sums $\sum_{i=n+1}^{\infty}\mathbb{Z}$ and we have the
following inverse system:
$$\sum_{1}^{\infty}\mathbb{Z}\hookleftarrow
\sum_{2}^{\infty}\mathbb{Z}\hookleftarrow
\sum_{3}^{\infty}\mathbb{Z}\hookleftarrow \cdots .$$ The inverse
limit of this system is zero and the space $X_1$ is acyclic.
However, since all homomorphisms in this system are nonzero
monomorphisms it follows that all fine coverings are nonacyclic.

\medskip
{\bf Example 5.2.} Consider the "compressed sinusoid" $CS$ as a
subspace of the rectangle $[0, 1]\times [-1, 1] \subset
\mathbb{R}^2:$
$$CS = \{(x, y)|\ y = \sin {\frac{1}{x}}\ \mbox{ if}\ x\in (0, 1],\ \mbox
{and}\ y\in [-1,1]\ \mbox{ if}\ x = 0 \}.$$

Let us remove the continuum $\{0\}\times [0, 1]$ from it. We get a
locally compact space $X_2 = CS\setminus (\{0\}\times [0, 1]).$

\begin{center}
\includegraphics*[trim=0cm 1cm 5cm 0cm, clip=true, width=8cm ]{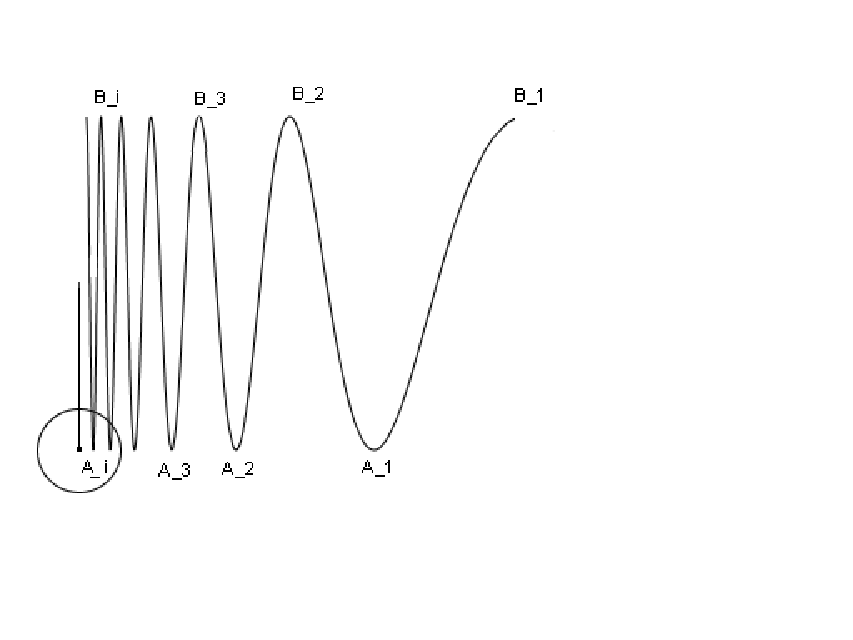}
\begin{center}\vspace{-3mm}\normalsize{Fig.2. Locally compact acyclic subspace of  $CS$ without acyclic fine  coverings.}\end{center}
\end{center}

\medskip

To prove the acyclicity consider the following strong deformation
retract $T$ of $CS\setminus (\{0\}\times [0, 1])$:
$$T = \{(x, y)|\ y = \sin {\frac{1}{x}}\ \mbox{ if}\ x\in (0, 1],\ \mbox
{and}\ y = -1\ \mbox{ if}\ x = 0 \}.$$

The space $T$ has the same homotopy type as $X_2$ and it consists
of the line $y = \sin {\frac{1}{x}}$\ which is homeomorphic to
$(0, 1]$ and point $(0; -1).$ There exists a cofinite system of
open coverings of this space $T$: On this line consider the
standard triangulation and its cover by open stars of its
vertices. For the point $(0; -1)$ consider the open subspace
$B((0; -1), \varepsilon)\cap T$ of $T.$ The nerves of this
cofinite system of coverings are homeomorphic to a countable
bouquet of circles and a segment.

We have (as in the first example) the cofinite inverse system of
homology groups and homomorphisms:
$$\sum_{1}^{\infty}\mathbb{Z}\hookleftarrow
\sum_{2}^{\infty}\mathbb{Z}\hookleftarrow
\sum_{3}^{\infty}\mathbb{Z}\hookleftarrow \cdots .$$

The inverse limit of this system is trivial, therefore
$\check{H}_1(T) = 0$ and hence space $X$ is acyclic with respect
to {\v C}ech homology groups. However, all fine covering of $X_2$
are nonacyclic.

\section{Epilogue}

It follows by Corollory \ref{contractible}  that every
$n-$dimensional contractible compactum has arbitrary fine
coverings of order $2n+1$ whose nerves  are all contractible.
The following question 
is a special case of Problem
\ref{AR}:

\begin{qws}
Does there exist an $n-$dimensional contractible compactum whose
fine coverings of order $n+1$ are all nonacyclic?
\end{qws}


\section{Acknowledgements}
This research was supported by the Slovenian Research Agency
grants J1-5435 and P1-0292. We are grateful to Robert Daverman and
Gerard Venema for their comments. We also thank the referee for remarks and suggestions.

\providecommand{\bysame}{\leavevmode\hbox to3em{\hrulefill}\thinspace}

\end{document}